\documentclass[11pt, a4paper,leqno]{amsart}
\usepackage{amsmath,amsthm,amscd,amssymb,amsfonts, amsbsy}
\usepackage{stmaryrd}
\usepackage[english]{babel}
\usepackage{latexsym}
\usepackage{txfonts}
\usepackage{exscale}
\usepackage{bbm}
\usepackage{enumitem}
\usepackage[colorlinks,citecolor=red,pagebackref,hypertexnames=false]{hyperref}
\usepackage{color}

\calclayout
\allowdisplaybreaks

\theoremstyle{plain}
\newtheorem{theorem}[equation]{Theorem}
\newtheorem{lemma}[equation]{Lemma}

\theoremstyle{definition}
\newtheorem{definition}[equation]{Definition}

\theoremstyle{remark}
\newtheorem{remark}[equation]{Remark}

\numberwithin{equation}{section}

\newcommand{\RR}{{\mathbb{R}}}
\newcommand{\eps}{\varepsilon}

\newcommand{\dist}{\operatorname{dist}}

\newcommand{\re}{\mathbb{R}}

\newcommand{\ree}{\mathbb{R}^{n+1}}

\newcommand{\dd}{\mathbb{D}}

\newcommand{\C}{\mathcal{C}}

\newcommand{\om}{\Omega}

\newcommand{\F}{\mathcal{F}}

\newcommand{\A}{\mathcal{A}}

\newcommand{\E}{\mathcal{E}}

\newcommand{\W}{\mathcal{W}}

\newcommand{\sbf}{{\bf S}}

\newcommand{\pom}{\partial\Omega}

\newcommand{\hm}{\omega}

\newcommand{\RNum}[1]{\uppercase\expandafter{\romannumeral #1\relax}}

\renewcommand{\emptyset}{\mbox{\textup{\O}}}

\DeclareMathOperator{\diam}{diam}

\begin{document}
\allowdisplaybreaks

\author{Simon Bortz}

\address{Simon Bortz,
	School of Mathematics, University of Minnesota, Minneapolis, MN
55455, USA}
	\email{bortz010@umn.edu} 
\author{Olli Tapiola}
\address{Olli Tapiola, Department of Mathematics, University of Missouri, Columbia, MO 65211, USA}
\email{olli.m.tapiola@gmail.com}

\thanks{The first author was supported by the NSF INSPIRE Award DMS-1344235. The second author was supported by Emil Aaltosen S\"a\"ati\"o through Foundations' Post Doc Pool grant.}
\title[$\eps$-Approximability in $L^p$ Implies Uniform Rectifiability]{$\eps$-Approximability of Harmonic Functions in $L^p$ Implies Uniform Rectifiability}
\begin{abstract}
Suppose that $\om \subset \RR^{n+1}$, $n \ge 2$, is an open set satisfying the corkscrew condition with an $n$-dimensional ADR boundary, $\pom$.
In this note, we show that if harmonic functions are $\eps$-approximable in $L^p$
for any $p > n/(n-1)$, then $\pom$ is uniformly rectifiable. Combining our results with those in \cite{HT} gives us a new characterization 
of uniform rectifiability which complements the recent results in \cite{HMM}, \cite{GMT} and \cite{AGMT}.
\end{abstract}

\maketitle

\section{Introduction}
The purpose of this note is to answer a question posed in \cite{HT}: If $E \subset \RR^{n+1}$ is an $n$-ADR set, does 
$\eps$-approximability of harmonic functions in $L^p$ for some fixed $p$ in $\RR^{n+1} \setminus E$ imply 
uniform rectifiability of $E$? We answer this question affirmatively with the following theorem.

\begin{theorem}\label{mainthrm.thrm} Let $\om \subset \ree$ be an open set with $n$-dimensional ADR boundary such 
that $\om$ satisfies the corkscrew condition. Let $p > n/(n-1)$ and suppose that every harmonic function 
is \emph{$\eps$-approximable in $L^p$} for every $\eps \in (0,1)$: there exist constants
$C_\eps$ and $C_p$ and a function $\varphi = \varphi^\eps \in BV_{\text{loc}}(\om)$ such that
$$\begin{cases}
\lVert N_*(u - \varphi) \rVert_{L^p(\pom, \sigma)} \le \eps C_p \lVert N_*u\rVert_{L^p(\pom, \sigma)},
\\ \lVert \C(\nabla \varphi) \rVert_{L^p(\pom, \sigma)} \le C_\eps C_p \lVert N_*u\rVert_{L^p(\pom, \sigma)}.
\end{cases}
$$
Then $\pom$ is uniformly rectifiable. Here $N_*u$ is the non-tangential maximal function (see Definition 
\ref{nontandef.def}) and 
$$\C(\nabla u)(x): = \sup_{r > 0} \frac{1}{r^n} \iint_{B(x,r) \cap \om} |\nabla \varphi| \, dY.$$
\end{theorem}
As usual, if $\varphi$ is not differentiable, $\iint_{(B(x,r) \cap \om} |\nabla \varphi| \, dY$ means the 
\emph{total variation of $\varphi$ in $B(x,r) \cap \om$},
\begin{align*}
  \iint_{B(x,r) \cap \om} |\nabla \varphi| \, dY 
  \coloneqq \sup_{\substack{\overrightarrow{\Psi} \in C_0^1(B(x,r) \cap \om) \\ \|\overrightarrow{\Psi}\|_{L^\infty(B(x,r) \cap \om) \le 1}}} \iint_{B(x,r) \cap \Omega} \varphi \, \text{div} \overrightarrow{\Psi} \, dY,
\end{align*}
and we denote $\varphi \in BV(\om)$ if the total variation over $\om$ is finite and $\varphi \in BV_{\text{loc}}(\om)$
if the total variation over any open relatively compact $\om' \subset \om$ is finite.

Our proof draws strongly on the ideas of \cite{GMT} and \cite{Hyt-Ros}. In particular, we will follow a central 
idea in \cite{GMT} and prove the following result:

\begin{theorem}
  \label{LpEpsimpUR.thrm}
  Suppose that the hypotheses of Theorem \ref{mainthrm.thrm} hold. 
  Then the harmonic measure $\omega$ admits a Corona decomposition (see Definition \ref{coronahm.def}). 
\end{theorem}
By \cite[Proposition 5.1]{GMT}, Theorem \ref{LpEpsimpUR.thrm} is enough to imply Theorem \ref{mainthrm.thrm}.

We provide some context to the results herein. There has been significant and continued interest in characterizations of 
geometric properties by properties of solutions to elliptic PDEs and/or elliptic measure. In light of this, 
a fundamental question at the interface of harmonic analysis and geometric measure theory is the following: 
what PDE properties serve to characterize uniform rectifiability of the boundary of an open set with ADR boundary? 
Recently, powerful tools from harmonic analysis and geometric measure theory have provided several characterizations 
\cite{HMM, GMT}. Among these is the following theorem.
\begin{theorem}[\cite{GMT,HMM}]
Let $\om \subset \ree$ be a domain satisfying the interior corkscrew condition with $n$-dimensional ADR boundary, $
\pom$. Then $\pom$ is uniformly rectifiable if and only if bounded harmonic functions in $\om$ are $\eps$-approximable 
(see Definition \ref{defin:eps_app}).
\end{theorem}

The direction uniform rectifiability implies $\varepsilon$-approximability was proved in \cite{HMM}, and the converse in \cite{GMT}.

A remarkable thing about this theorem, is that it does not have any assumptions on the connectivity of $\om$ or $\pom$. For a long time, giving up connectivity assumptions was a serious obstacle in the field since uniform rectifiability is not 
enough to imply absolute continuity of the elliptic measure with respect to the surface measure \cite{BJ}. Without 
absolute continuity, analyzing the properties of the solutions to elliptic PDE becomes considerably more difficult.

Let us recall the definition of the usual $L^\infty$ type $\eps$-approximability:
\begin{definition}
  \label{defin:eps_app}
  Suppose that $\om \subset \ree$ is a domain satisfying the interior corkscrew condition with $n$-dimensional ADR  
boundary $\pom$  (see Definition \ref{defadr}) and let $\eps \in (0,1)$. We say that a function $u$ is $\eps$-
approximable if there exists a constant $C_\eps$ and 
  a function $\varphi = \varphi^\eps \in BV_{\text{loc}}(\Omega)$ satisfying
  \begin{align*}
    \|u-\varphi\|_{L^\infty(\Omega)} < \eps \ \ \ \ \ \text{ and } \ \ \ \ \ \sup_{x \in E, r > 0} \frac{1}{r^n} \iint_{B(x,r) \cap 
\Omega} |\nabla \varphi(Y)| \, dY \le C_\eps.
  \end{align*}
\end{definition}
The notion of $\eps$-approximability was first introduced by Varopoulos who proved that every harmonic function in $\RR^{n+1}_+$ is 
$\eps_0$-approximable for some $\eps_0 \in (0,1)$ \cite{V}. He used the property to prove the so called Varopoulos extension 
theorem which gives an alternative characterization of $BMO$ functions. In his work, it was not necessary to have the 
approximability property for all $\eps \in (0,1)$ but in later developments a sharper version of the property has been 
crucial. Garnett \cite{G} extended Varopoulos's result for all $\eps \in (0,1)$ and his result in turn was generalized 
to bounded Lipschitz domains by Dahlberg \cite{D}. The property has been used to e.g. explore the absolute continuity properties of elliptic measures\footnote{See also \cite{G}, where similar ideas were introduced in the formulation and proof of a quantitative Fatou theorem.} \cite{KKPT, HKMP}  and, 
as stated above, give a new characterization of uniform rectifiability \cite{HMM, GMT}. The $L^p$ version of $\eps$-approximability 
was recently introduced by Hyt\"onen and Ros\'en \cite[Theorem 1.3]{Hyt-Ros} who showed 
that any weak solution to certain elliptic partial differential equations in $\ree_+$ are $\eps$-approximable in 
$L^p$ for every $\eps \in (0,1)$ and every $p \in (1,\infty)$. More recently, Steve Hofmann and the second author 
proved an analog of \cite[Theorem 1.3]{Hyt-Ros} for harmonic functions in the rougher setting of a domain with uniformly rectifiable boundary 
which satisfies the interior corkscrew condition. We note that the result in \cite{HT} is 
stated for the complement of an ADR set and in that setting the interior corkscrew condition is automatically satisfied.
However, it is straightforward to show that the result holds in the current context, too.

Combining our main result with results in \cite{HMM, HMM2, GMT} and \cite{HT} gives us the following characterization 
result:
\begin{theorem}\label{theorem_conditions}
  Let $\om \subset \ree$, $n \ge 2$, be an open set with $n$-dimensional ADR boundary such that $\om$ satisfies the corkscrew condition.
  Then the following conditions are equivalent.
  \begin{enumerate}
    \item[(a)] $\pom$ is uniformly rectifiable.
    \item[(b)] Every bounded harmonic function $u$ satisfies the following Carleson measure estimate:
               \begin{align*}
                 \sup_{x \in \pom, r > 0} \frac{1}{r^n} \iint_{B(x,r) \setminus \pom} |\nabla u(Y)|^2 \dist(Y,\pom) \, dY \lesssim \|u\|_{L^\infty(\om)}^2.
               \end{align*}
    \item[(c)] For any fixed $p \in [2,\infty)$, every function $u \in C_0(\overline{\om})$ that is harmonic in $\om$ satisfies 
               the following $S \lesssim N$ estimate:
               \begin{align*}
                 \|Su\|_{L^p(\pom)} \le C_p \|N_* u\|_{L^p(\pom)}.
               \end{align*}
    \item[(d)] Every bounded harmonic function on $\om$ is $\eps$-approximable for all $\eps > 0$.
    \item[(e)] For any fixed $p > n/(n-1)$, every harmonic function on $\om$ is $\eps$-approximable in $L^p$ for all $\eps > 0$.
  \end{enumerate}
\end{theorem}
We remark that at least the conditions (a), (b) and (d) remain equivalent for if we assume the conditions both for certain 
types of elliptic operators and their adjoints as was proven in \cite{AGMT}. The methods we use to prove 
Theorem \ref{mainthrm.thrm} should extend to the class elliptic operators considered in \cite{AGMT}, but we choose to treat 
the case of the Laplacian for the sake of brevity. 
We note that in the proof of Theorem \ref{mainthrm.thrm} we need the approximability property only for some sufficiently small 
$\eps_1 \in (0,1)$ depending on $p$ and the structural constants. The restriction $p > n/(n-1)$ is due 
to the proof of Lemma \ref{Lem2.lem}; a suitable version of the Serrin-Weinberger theory \cite{SW} would allow one to reach $p =  n/(n-1)$ (see \cite[Proposition 5.1]{Hyt-Ros}).

\section{Preliminaries and Definitions}

\begin{definition}[ADR (Ahlfors-David regular) sets]\label{defadr}
We say that a  set $E \subset \ree$, of Hausdorff dimension $n$, is \emph{ADR}
if it is closed, and if there is some uniform constant $C$ such that
\begin{equation} \label{eq1.ADR}
\frac1C\, r^n \leq \sigma\big(\Delta(x,r)\big)
\leq C\, r^n,\quad\forall r\in(0,\diam (E)),\ x \in E,
\end{equation}
where $\diam(E)$ may be infinite.
Here, $\Delta(x,r):= E\cap B(x,r)$ is the ``surface ball" of radius $r$,
and $\sigma:= H^n|_E$ 
is the ``surface measure" on $E$, where $H^n$ denotes $n$-dimensional
Hausdorff measure.
\end{definition}

\begin{definition}[UR]
  \label{defin:ur}
  Following \cite{DS1, DS2}, we say that an ADR set $E \subset \RR^{n+1}$ is \emph{UR} (uniformly rectifiable) 
  if it contains ``big pieces of Lipschitz images'' (BPLI) of $\RR^n$:
  there exist constants $\theta, M > 0$ such that for every $x \in E$ and $r \in (0,\diam(E))$ there is a Lipschitz mapping
  $\rho = \rho_{x,r} \colon \RR^n \to \RR^{n+1}$, with Lipschitz norm no larger that $M$, such that
  \begin{align*}
    H^n(E \cap B(x,r) \cap \rho(\{y \in \RR^n \colon |y| < r\})) \ge \theta r^n.
  \end{align*}
\end{definition}

\begin{definition}[Corkscrew Condition]\label{CS.def}
We say that an open set $\om \subset \ree$ satisfies the \emph{corkscrew condition} if there exists a constant $M$ such that for 
every $x \in \pom$ and $r\in(0,\diam(\pom))$ there exists a point $Y \in \om$ such that $B(Y, \tfrac{r}{M}) \subset B(x,r) 
\cap \om$. If $Y$ is as above we say $Y$ is a corkscrew point relative to $x$ at scale $r$. 
\end{definition}

Our standing assumption will be that $\om \subset \ree$ is an open set with $n$-dimensional ADR 
boundary $\pom$ and that $\om$ satisfies the corkscrew condition with constant $M$. Let us fix some notation:
\begin{enumerate}
  \item[$\bullet$] We will use lowercase letters $x$, $y$, $z \ldots$ to denote points on $\pom$ and capital letters 
                   $X, Y, Z, \ldots$ to denote generic points in $\RR^{n+1}$.
                   
  \item[$\bullet$] We use the standard notation $B(X,r)$ for usual Euclidean balls in $\RR^{n+1}$ and $\Delta(x,r) \coloneqq B(x,r) \cap \pom$
                   for surface balls. As usual, we use the notation
                   $\kappa B(X,r) \coloneqq B(X,\kappa r)$ and $\kappa \Delta(x,r) \coloneqq \Delta(x,\kappa r)$ for the
                   dilations of the balls.
                   
  \item[$\bullet$] We set $\delta(X) \coloneqq \dist(X,\pom)$ for every $X \in \om$.
  
  \item[$\bullet$] We let $\hm = \hm^X$ to denote the harmonic measure for $\om$.
\end{enumerate}

We will use actively well-known dyadic techniques on $\pom$. We note that in our context there is no canonical choice for 
the dyadic system but it is not necessary to know the exact structure of the ``cubes''.

\begin{lemma}[Dyadic systems for ADR sets]\label{lemmaCh}
Suppose that $E\subset \ree$ is closed $n$-dimensional ADR set.  Then there exist
constants $a_0>0$ and $a_1<\infty$, depending only on dimension and the
ADR constant, such that for each $k \in \mathbb{Z}$ there is a collection of Borel sets (``cubes'')
$$
\mathbb{D}_k:=\{Q_{j}^k\subset E: j\in \mathfrak{I}_k\},$$ where
$\mathfrak{I}_k$ denotes some (possibly finite) index set depending on $k$, satisfying

\begin{enumerate}
\item $E=\cup_{j}Q_{j}^k\,\,$ for each $k\in{\mathbb Z}$ and the union is disjoint.

\item If $m\geq k$ then either $Q_{i}^{m}\subset Q_{j}^{k}$ or $Q_{i}^{m}\cap Q_{j}^{k}=\emptyset$.

\item For each $(j,k)$ and each $m<k$, there is a unique $i$ such that $Q_{j}^k\subset Q_{i}^m$.

\item For each $(j,k)$ there exists a point $x_j^k \in Q_j^k$ such that
      \begin{align*}
        \Delta(x_j^k, a_0 2^{-k}) \subset Q_j^k \subset \Delta(x_j^k, a_1 2^{-k}) \coloneqq B(x_j^k,a_1 2^{-k}) \cap E.
      \end{align*}
\end{enumerate}
\end{lemma}
In the literature, there exist numerous proofs for this result with many additional properties (see e.g. \cite{DS1,DS2,Ch,HK,HMMM}) but we listed only
the ones that we actually need in this paper. Let us fix some notation related to the dyadic system.
\begin{enumerate}
  \item[$\bullet$] If the set $E$ is bounded, then we simply have $Q_j^k = E$ for all sufficiently small $k \in \mathbb{Z}$.
                   Because of this, we denote $\dd \coloneqq \bigcup_k \dd_k$, where the union runs over the $k$ such that 
                   $2^{-k} \lesssim  {\rm diam}(E)$.
                   
  \item[$\bullet$] For each $Q \coloneqq Q_j^k \in \dd$ we denote
                   \begin{align*}
                     x_Q \coloneqq x_j^k, \qquad B_Q \coloneqq B(x_Q,a_1 2^{-k}), \qquad \Delta_Q \coloneqq B_Q \cap E.
                   \end{align*}
                   We shall refer to the point $x_Q$ as the ``center'' of $Q$.
                   
  \item[$\bullet$] For each $Q \coloneqq Q_j^k \in \dd$, we set $\ell(Q) \coloneqq 2^{-k}$.
                   We refer to this quantity as the ``side length'' of $Q$. Since we assume that $\ell(Q) \lesssim \diam(E)$
                   for every $Q \in \dd$, we have $\ell(Q)\approx \diam(Q)$ and $\ell(Q)^n \approx \sigma(Q)$.
  
  \item[$\bullet$] For each $R \in \dd$, we let $\dd_R \subset \dd$ be the collection of the subcubes of $R$.
\end{enumerate}

Next, we define ``dyadic'' Whitney regions similar to those in \cite{HM-UR1}: for each $Q \in \dd$ we choose a bounded number 
of usual Whitney cubes $I_Q^i \subset \om$ such that $\bigcup_i I_Q^i$ acts as a substitute for a region of the 
type $P \times (\ell(P)/2, \ell(P)]$, the standard Whitney region in the upper half-space. Our Whitney regions are not disjoint but they 
have a bounded overlap property which is enough for us. In \cite{HM-UR1,HMM,HT}, it was crucial to fatten the cubes $I_Q^i$ 
to ensure that most of the regions break into components with strong geometric properties (see particularly 
\cite[Lemma 3.24]{HMM}). In our situation, we do not have to fatten the cubes but we note that we have to make sure that the Whitney 
regions are large enough for the non-tangential maximal function we use (see Definition \ref{nontandef.def})
to be equivalent in $L^p$ sense with the non-tangential maximal function used it \cite{HT}. See \cite[Lemma 1.23]{HT} and its proof 
for more details about this and corresponding Fefferman-Stein \cite{FS} type arguments.

Let us be more precise. Let $\eta \ll 1 \ll K$ be parameters depending on $n$, $M$ and the ADR constant whose values we will determine 
later in the proofs. Suppose that $\W \coloneqq \{I\}_I$ is a Whitney decomposition of $\om$, that is, $\{I\}_I$ is a collection of 
closed $(n+1)$-dimensional Euclidean cubes whose interiors are disjoint such that $\bigcup_I I = \ree \setminus \pom$ and
\begin{equation*}
4 \diam(I)\leq
\dist(4I,\pom)\leq \dist(I,\pom) \leq 40\diam(I)\,,\qquad \forall\, I\in \mathcal{W}
\end{equation*}
and 
$$1/4 \diam(I_1) \le \diam(I_2) \le 4 \diam(I_1)$$
whenever $I_1 \cap I_2 \neq \emptyset$. For every $Q \in \dd(\pom)$ we set
$$W_Q(\eta, K) = \{I \in \W: \eta^{1/4} \ell(Q) \le \ell(I) \le K^{1/2}\ell(Q), \dist(I,Q) \le K^{1/2} \ell(Q)\}.$$
and
$$U_Q(\eta,K) = \bigcup_{I \in W_Q(\eta, K)} I.$$
We call $U_Q(\eta,K)$ the Whitney region relative to $Q \in \dd$.

\begin{remark}\label{choiceEtaK.rmk}
We put some initial restrictions on the parameters $\eta$ and $K$ to ensure that $U_Q(\eta,K)$ is non-empty for every 
$Q \in \dd$. We notice that by the corkscrew condition for every $Q \in \dd$ there exists a corkscrew 
point $X_Q$ such that $|X_Q - x_Q| < \ell(Q)$ and $\dist(X_Q, \pom) > M^{-1} \ell(Q)$. It follows that $X_Q \in I$ for 
some $I \in \W$ with $$\dist(I,Q) \approx \ell(Q).$$ Choosing $\eta \ll 1 \ll K$ depending on the corkscrew condition and 
$n$ we obtain that $U_Q(\eta,K) \neq \emptyset$. We later impose further assumptions on $\eta,K$, but these 
will continue to only depend on $n$, the ADR constant and the corkscrew condition. We will drop the $\eta, K$ from 
$U_Q(\eta,K)$ for notational convenience.
\end{remark}

Finally, let us define cones and the non-tangential maximal operator. We first denote the usual cone 
of aperture $\alpha > 1$ at $x \in \pom$ by $\widetilde{\Gamma}_\alpha(x) \coloneqq \{Y \in \om \colon |x-Y| < \alpha \cdot \delta(Y)\}$.
We notice that if $\pom$ is bounded, then we have $\RR^{n+1} \setminus B(x,R) \subset \widetilde{\Gamma}_\alpha(x)$ for 
large enough $R$ and every $x \in \pom$. Thus, since we only construct the regions $U_Q$ for such $Q$ that $\ell(Q) \lesssim \diam(\pom)$,
we set
\begin{align*}
  \Gamma(x) \coloneqq \left\{ \begin{array}{cl}
                                \bigcup_{Q \in \dd, Q \ni x} U_Q, &\text{ if } \diam(\pom) = \infty \\
                                \bigcup_{Q \in \dd, Q \ni x} U_Q \cup (\RR^{n+1} \setminus B(x,R_0)), &\text{ if } \diam(\pom) < \infty
                              \end{array} \right. ,
\end{align*}
for a fixed large enough $R_0$. It is straightforward to verify that by choosing the constants $\eta$ and $K$ in a suitable 
way, there exist $\alpha _1 > \alpha_0 > 1$ such that $\widetilde{\Gamma}_{\alpha_0}(x) \subset \Gamma(x) \subset \widetilde{\Gamma}_{\alpha_1}(x)$ 
for every $x \in \pom$.

\begin{definition}[Non-tangential maximal function]\label{nontandef.def}
For any function $g: \om \to \re$ we define the non-tangential maximal function $N_*g : \pom \to \re$ by
$$N_*g(y):= \sup_{X \in \Gamma(y)} |g(X)|.$$
\end{definition}

\section{Two lemmas}

In order to present the proof of Theorem \ref{LpEpsimpUR.thrm} we prove two simple lemmas. 

\begin{lemma}\label{Lem1.lem}
Let $Q \in \dd$. Suppose that $f$ is a Borel function such that $|f| \le 1_Q$ and set
$u(X) = \int_{\pom} f(y) \, d\hm^X(y)$ for every $X \in \om$. Then
$$|u(X)| \lesssim 1_{4B_Q}(X) + 1_{(4B_Q)^c}(X) \left(\frac{\ell(Q)}{|X - x_Q|} \right)^{n -1},$$
where the implicit constants depend on $n$ and the ADR constant. 
\end{lemma}

\begin{proof}
Let us set
$$H(X) = \frac{1}{\ell(Q)} \int_{B_Q} \E(X,y) \, d\sigma(y),$$
where $$\E(X,Y) := \frac{c_n}{|X -Y|^{n-1}}$$
is the fundamental solution to the Laplacian in $\ree$. By the ADR condition and the local $\sigma$-integrability of 
$\E$, we know that $H$ is bounded in $\ree$ and we have
\begin{align}
  \label{lower_bound_on_Q} H(y) \gtrsim 1, \quad \forall y \in Q.
\end{align}
We also notice that
$$0 \le H(X) \lesssim \left(\frac{\ell(Q)}{|X - x_Q|} \right)^{n -1},\quad \forall X \in (4B_Q)^c,$$
where the implicit constants above depend on $n$ and the ADR constant.
It is straightforward to show that $H$ is harmonic in $\om$ and continuous in $\overline{\om}$. Thus, we have
$$H(X) = \int_{\pom} H(y) \, d \hm^X(y).$$
In particular, since $|f| \le 1_Q$, we get
$$|u(X)| \le \int_{Q} |f(y)| \, d\hm^X(y) \overset{\eqref{lower_bound_on_Q}}{\lesssim} H(X) \lesssim 1_{4B_Q}(X) + 1_{(4B_Q)^c}(X) \left(\frac{\ell(Q)}{|X - x_Q|} \right)^{n -1}$$
by the boundedness of $H$.
\end{proof}

This lemma readily yields a bound also for the non-tangential maximal operator acting on functions of the same type:
\begin{lemma}\label{Lem2.lem}
Let $Q \in \dd$. Suppose that $f$ is a Borel function satisfying $|f| \le 1_Q$ an set 
$u(X) = \int_{\pom} f(y)\, d \hm^X(y)$ for every $X \in \om$. Then for all $p > n/(n -1)$ we have
  $$\lVert N_*u \rVert_{L^p(\pom, \sigma)} \le C_1 \sigma(Q)^{1/p},$$
where $C_1$ depends on $n$, the ADR constant, $\eta$, $K$ and $p$.
\end{lemma}

\begin{proof}
Let $y \in \pom$ and suppose that $X \in \Gamma(y)$. Then by the definition of $\Gamma(y)$ we have that
$$|X - y| \approx_{\eta,K} \delta(X) \le |X - x_Q|.$$
It follows that 
$$|y - x_Q| \le |X - y| + |X - x_Q| \lesssim |X - x_Q|.$$
Thus, Lemma \ref{Lem1.lem} yields
\begin{align*}
  |u(X)| &\lesssim 1_{4B_Q}(X) + 1_{(4B_Q)^c}(X) \left(\frac{\ell(Q)}{|X - x_Q|} \right)^{n -1} \\
         &\lesssim 1_{4B_Q}(X) + 1_{(4B_Q)^c}(X) \left(\frac{\ell(Q)}{|y - x_Q|} \right)^{n -1}, \quad \forall X \in \Gamma(y).
\end{align*}
Let us set $A_k := 2^{k+1} \Delta_Q \setminus 2^k \Delta_Q$ for every $k \ge 2$. By the ADR property, we get
\begin{align*}
\int_{\pom} N_*u^p \, d\sigma &\le \int_{4\Delta_Q} N_*u^p \, d\sigma + \sum_{k =2}^\infty \int_{A_k} N_*u^p \, d\sigma \\
                              &\lesssim \sigma(4\Delta_Q) + \sum_{k=2}^\infty \int_{A_k} \left( \frac{\ell(Q)}{|y-x_Q|} \right)^{p(n-1)} \, d\sigma \\
                              &\lesssim \sigma(Q) + \sum_{k=2}^\infty \int_{A_k} \frac{1}{2^{k(n-1)p}} \, d\sigma \\
                              &\lesssim \sigma(Q) + \sigma(Q) \sum_{k=2}^\infty 2^{kn - k(n-1)p}
                              \lesssim \sigma(Q),
\end{align*}
since $p > n/(n-1)$. 
\end{proof}

\section{Corona decomposition for harmonic measure and the proof of Theorem \ref{LpEpsimpUR.thrm}}

In this section, we prove Theorem \ref{LpEpsimpUR.thrm} which, as we pointed out earlier, is enough to imply Theorem \ref{mainthrm.thrm}.
Before the proof, we recall some definitions and results from \cite{GMT}.

\begin{definition}\cite{DS2}.\label{d3.11}   
Let $\sbf\subset \dd$. We say that $\sbf$ is \emph{coherent} if the following conditions hold:
\begin{enumerate}

\item[(a)] $\sbf$ contains a unique maximal element $Q(\sbf)$ which contains all other elements of $\sbf$ as subsets.

\item[(b)] If $Q$ belongs to $\sbf$ and $Q\subset Q' \subset Q(\sbf)$ for any $Q' \in \dd$, then $Q' \in {\bf S}$.

\item[(c)] Given a cube $Q_j^k \in \sbf$, either all of its children (i.e. the cubes $P \in \dd_{k+1}$ such that $P \subset Q_j^k$) 
             belong to $\sbf$, or none of them do.
\end{enumerate}
\end{definition}

\begin{definition}[Corona decomposition for harmonic measure \cite{GMT}]\label{coronahm.def}
Let $\om \subset \ree$ be an open set satisfying the corkscrew condition with $n$-dimensional ADR boundary, and let $\hm$ 
be the harmonic measure for $\om$. We say that \emph{$\hm$ admits a corona decomposition} if $\dd$ can be decomposed into disjoint coherent 
subcollections $\sbf$ such that the following two conditions holds.
\begin{enumerate}
  \item[(1)] The maximal cubes, $Q(\sbf)$, satisfy a Carleson packing condition
             $$\sum_{Q(\sbf) \subset R} \sigma(Q(\sbf)) \le C \sigma(R), \quad \forall R \in \dd(\pom).$$

  \item[(2)] For each $Q(\sbf)$ there exists $P_{Q(\sbf)} \in \om$ such that 
             $$c^{-1} \ell(Q(\sbf)) \dist(P_{Q(\sbf)},Q(\sbf)) \le \dist(P_{Q(\sbf)},\pom) \le c \ell(Q(\sbf)),$$
             $$\hm^{P_{Q(\sbf)}}(3R) \approx \frac{\sigma(R)}{\sigma(Q(\sbf))} \quad \forall R \in \sbf,$$
             where the implicit constants above and $c$ are uniform in $\sbf$ and $R$.
\end{enumerate}
\end{definition}

Let us fix the value of $\eps \ll 1$ later. For each cube in $\dd$, let $P_Q \in B_Q$ be a corkscrew point
at scale $\eps \ell(Q)$ relative to $x_Q$. Then we have
$$\frac{\eps}{M} \ell(Q) \le \delta(P_Q)  \le \eps(\ell(Q)),$$
Let $y_Q \in \pom$ be the touching point for the point $P_Q$, that is, $|y_Q - P_Q| = \delta(P_Q)$. 
By choosing $\eps \ll 1$ to be small enough, we may assume that
$$B(y_Q, |y_Q - P_Q|) \subset B(x_Q,\tfrac{3}{4}a_0\ell(Q))\footnote{Note that $B(x_Q,a_0\ell(Q))$ is exactly ``$B_Q$" from \cite{GMT}. See Lemma \ref{lemmaCh} (4).}.$$
Next, for some parameter $\tau$ to be chosen later depending on $M$, dimension and the ADR constant, we set
\begin{align*}
  S_Q   &\coloneqq y_Q + 2\tau(P_Q - y_Q), \\
  V_Q   &\coloneqq B(P_Q, (1-\tau)\delta(P_Q)), \\
  V_Q^1 &\coloneqq B(P_Q, (1/2)\delta(P_Q)), \\
  V_Q^2 &\coloneqq B(S_Q, \tau\delta(P_Q)).
\end{align*}
Notice that $V_Q^1, V_Q^2 \subset V_Q$. We then have:

\begin{lemma}[{\cite[Lemma 3.3, proof of Lemma 3.7]{GMT}}] \label{Lem3.3AGMT.lemma}
  Suppose that $\tau \ll 1$ and $\eps \ll \tau$ are chosen appropriately depending on $M$, dimension and the ADR constant. Then 
  if $E_Q \subset Q \in \dd$ is such that 
  \begin{equation*} 
    \hm^{P_Q}(E_Q) \ge (1- \eps) \hm^{p_Q}(Q),
  \end{equation*}
  there exists a non-negative harmonic function $u_Q$ on $\om$ and a Borel function $f_Q$ with 
  \begin{equation*}
    u_Q(X) = \int_{\pom} f_Q \, d\hm^X, \quad 0 \le f_Q \le 1_{E_Q}
  \end{equation*}
  satisfying
  \begin{equation} \label{bigosc.eq}
    |m_{V_Q^1} u_Q - m_{V_Q^2} u_Q| \ge c_1,
  \end{equation}
  where $c_1$ depends only on $M$, $n$ and the ADR constant and $m_{B_i} u = \fint_{B_i} u \, dX$.
\end{lemma}

\begin{definition}[Low density cubes \cite{GMT}]
  Let $0 < \delta \ll 1$ be a fixed constant. 
  For a cube $R \in \dd$ we say that a subcube $Q \in \dd_R$ is a \emph{low density cube} and write $Q \in \text{LD}(R)$ if $Q$ is a maximal cube (with resepect to containment) satisfying
  \begin{equation*}
    \frac{\hm^{P_{R}}(Q)}{\sigma(Q)} \le \delta \frac{\hm^{P_{R}}(R)}{\sigma(R)}.
  \end{equation*}
  For any cube $R \in \dd$, we denote $\text{LD}^0(R) = \{R \}$ and define $\text{LD}^k(R)$, $k \ge 1$, inductively by 
  $$\text{LD}^k(R) = \bigcup_{Q \in \text{LD}^{k-1}(R)} \text{LD}(Q).$$
\end{definition}

In the proof of the corona decomposition for harmonic measure in \cite{GMT}, $\eps$-approximability is used only to prove a 
packing condition for the low density cubes \cite[Lemma 3.7]{GMT}. Thus, we actually have:

\begin{lemma} \label{LDpack.lemma}
  Suppose that for any $m \ge 1$ and $R \in \dd$ we have 
  \begin{equation}\label{LDpackcl.eq1}
    \sum_{k = 0}^m \sum_{Q \in LD^k(R)} \sigma(Q) \le C \sigma(R),
  \end{equation}
  where $C$ is independent of $m$ and $R$. Then $\hm$ admits a corona decomposition.
\end{lemma}

We now prove Theorem \ref{LpEpsimpUR.thrm}.

\begin{proof}[Proof of Theorem \ref{LpEpsimpUR.thrm}] 
Let us start by fixing $\eta$ and $K$ depending on $\tau$ and $\eps$ so that
\begin{equation}\label{VQcontainment.eq}
V_Q \subset U_Q(\eta,K).
\end{equation}
This can be done since every point $Y \in V_Q$ has the property that 
$$\delta(Y) \approx_{\eps, \tau} \ell(Q).$$
As $\eps$ and $\tau$ depend only on $M$, $n$ and the ADR constant so do $\eta$ and $K$.
We also note that by the construction of the regions $U_Q(\eta,K)$ we have
\begin{equation}\label{VQoverlap.eq}
\sum_{Q \in \dd} 1_{V_Q}(X) \lesssim 1, \quad \forall X \in \om,
\end{equation}
where the implicit constant depends on $\eta, K$, $M$, $n$ and the ADR constant.

By Lemma \ref{LDpack.lemma}, to prove the theorem it is enough to show \eqref{LDpackcl.eq1}. We do this now. For every cube $Q \in \dd$ and every $m \ge 1$ we set
  $$\A(Q,m) \coloneqq \frac{1}{\sigma(Q)}\sum_{Q' \in \cup_{k =0}^m \text{LD}^{k}(Q)}\sigma(Q').$$
Then \eqref{LDpackcl.eq1} is equivalent to the statement
  $$\A(m) \coloneqq \sup_{Q \in \dd} \A(Q,m) \le C,$$
where $C$ is independent of $m$.

Let us fix $R \in \dd$ and set $\F_{1,m} \coloneqq \bigcup_{k =1}^m \text{LD}^{k}(R)$.
For $Q \in \F_{1,m}$ we set $L_Q: = \cup_{Q' \in \text{LD}(Q)} Q'$ and 
  $$E_Q := Q \setminus L_Q.$$ 
The sets $\{E_Q\}_{Q \in \F_{1,m}}$ are pairwise disjoint by definition. Moreover, by the definition of $\text{LD}(Q)$, we have
  $$\hm^{P_Q}(L_Q) \le \sum_{Q' \in \text{LD}(Q)} \hm^{P_Q}(Q') \le \delta \sum_{Q' \in \text{LD}(Q)} \frac{\sigma(Q')}{\sigma(Q)} 
    \hm^{P_Q}(Q) \le \delta \hm^{P_Q}(Q)$$
and hence
\begin{equation}\label{LDpackcl.eq3}
  \hm^{P_Q}(E_Q) \ge (1-\delta)\hm^{P_Q}(Q) \ge (1 -\eps)\hm^{P_Q}(Q)
\end{equation} 
as long as we choose $\delta \le \eps$. Then we may apply Lemma \ref{Lem3.3AGMT.lemma} to obtain a collection of functions $\{u_Q\}_{Q \in 
\F_{1,m}}$ such that
  $$u_Q(X) = \int_{\pom} f_Q \, d\hm^X, \quad 0 \le f_Q \le 1_{E_Q}$$
for some Borel function $f_Q$ and
  $$|m_{V_Q^1} u_Q - m_{V_Q^2} u_Q| \ge c_1.$$
Let $\eps_1 > 0$ to be chosen. Let $\Xi$ denote the collection of sequences $\{a = (a_Q): Q \in \F_{1,m}, a_Q \in \{-1,+1\}\}$ and let $\lambda$ be a 
probability measure on $\Xi$ which assigns equal probability to $-1$ and $+1$. For every $a \in \Xi$ we set
  $$u_a(X) = \sum_{Q \in \F_{1,m}} a_Q u_Q(X).$$
Note that since $f_Q$ are Borel functions with disjoint supports, $\sum_{Q \in \F_{1,m}} a_Qf_Q$ is a Borel function
and clearly 
  $$|u_a(X)| \le \int_{\pom} \sum_{Q \in \F_{1,m}} |a_Q|f_Q \, d\hm^X \le \sum_{Q \in \F_{1,m}} \hm^X(E_Q) \le 1, \quad \forall X \in \om.$$
We now apply the $\eps$-approximability in $L^p$ property (see Theorem \ref{mainthrm.thrm}) with
$``\eps" = \eps_1$: for each $a \in \Xi$ there exists $\varphi_a \in BV_{\text{loc}}$ such that 
$$\begin{cases}
  \lVert N_*(u_a - \varphi_a) \rVert_{L^p(\pom, \sigma)} \le \eps_1 C_p \lVert N_*u_a\rVert_{L^p(\pom, \sigma)} \\
  \lVert \C(\nabla \varphi_a) \rVert_{L^p(\pom, \sigma)} \le C_{\eps_1} C_p \lVert N_*u_a\rVert_{L^p(\pom, \sigma)}.
\end{cases}$$
By Lemma \ref{Lem2.lem} we also have
\begin{equation}\label{Ntboundwcube.eq}
  \lVert N_*(u_a - \varphi_a) \rVert_{L^p(\pom, \sigma)} \le  \eps_1 \lVert N_*u_a\rVert_{L^p(\pom, \sigma)} \le C_1 
  \eps_1 \sigma(R)^{1/p}
\end{equation}
and 
\begin{equation}\label{CMEvarphi.eq}
  \lVert \C(\nabla \varphi_a) \rVert_{L^p(\pom, \sigma)} \le C_{\eps_1}C_1\sigma(R)^{1/p},
\end{equation}
where $C_1$ depends on $n$, the ADR constant, $\eta$, $K$ and $p$. By Chebyshev's inequality and 
\eqref{Ntboundwcube.eq}, for each $a \in \Xi$ we have
  $$\sigma(\{x \in R: N_*(u_a - \varphi_a)(x) > C_2\eps_1 \}) \le \frac{C_1^p}{C_2^p}\sigma(R).$$
We will fix the exact value of $\gamma \in (0,1)$ later but regardless of the exact value, we may choose
$C_2$ so that $C_1^p / C_2^p < \gamma/2$. Let us set $\eps_0 := C_2 \eps_1$. It follows that for each 
$a \in \Xi$ there exists a set $F(R,a) \subset R$ such that $\sigma(F(R,a)) > (1-\gamma) \sigma(R)$ 
and for all $y \in F(R,a)$
$$|u_a(X) - \varphi_a(X)| \le \eps_0, \quad \forall X \in \Gamma(y).$$
Now for each $a \in \Xi$ we let $\F_{2,m}(a)$ be the collection of cubes $Q \in \F_{1,m}$ such that $Q \cap F(R,a) = 
\emptyset$. We then let $\widetilde{\F}_{2,m}(a)$ be the collection of maximal cubes in $\F_{2,m}(a)$ with respect to 
inclusion and $\F_{3,m}(a) = \F_{1,m} \setminus \F_{2,m}(a)$. By maximality,
\begin{equation}\label{AmpleSawtooth.eq}
  \sum_{Q^* \in \widetilde{\F}_{2,m}(a)} \sigma(Q^*) \le \gamma \sigma(R).
\end{equation}
Suppose that $Q \in \F_{3,m}(a)$. Then there exists $y \in F(R,a) \cap Q$.  It follows that $U_Q \subset \Gamma(y)$ 
and hence 
  $$|u_a(X) - \varphi_a(X)| \le \eps_0, \quad \forall X \in U_Q.$$
In particular, by \eqref{VQcontainment.eq} we have for all $Q \in \F_{3,m}(a)$
\begin{equation}\label{epsinsawtooth.eq}
  |u_a(X) - \varphi_a(X)| \le \eps_0, \quad \forall X \in V_Q.
\end{equation}

Now, using \eqref{bigosc.eq} and Khintchine's inequality we obtain for every $Q \in \F_{1,m}$
\begin{equation}\label{khintchine.estimate}
\begin{split}
  c_1 &\le |m_{V_Q^1} u_Q - m_{V_Q^2} u_Q| \\
      &\le \left(\sum_{Q' \in \F_{1,m}} |m_{V_{Q}^1} u_{Q'} - m_{V_{Q}^2}u_{Q'}|^2 \right)^{1/2} \\
      &\le \frac{1}{c_3} \int_{\Xi} \left| \sum_{Q' \in \F_{1,m}} a_{Q'}(m_{V_{Q}^1}u_{Q'} - m_{V_{Q}^2}u_{Q'})\right| \, d\lambda(a) \\
      &= \frac{1}{c_3} \int_{\Xi}|m_{V_Q^1} u_a - m_{V_Q^2} u_a|\, d\lambda(a),
  \end{split}
\end{equation}
where $c_3$ is a universal constant.
Then it follows that integrating over $V_Q$ (note that $\diam(V_Q) \approx \ell(Q)$) and summing in $Q$ we have
\begin{equation}\label{bigest1.eq}
  \begin{split}
    \sum_{Q \in \F_{1,m}} \sigma(Q) &\lesssim \sum_{Q \in \F_{1,m}} \ell(Q)^n \\
                                    &\overset{\eqref{khintchine.estimate}}{\lesssim} \sum_{Q \in \F_{1,m}} \int_{V_Q}  \int_{\Xi} \frac{1}{\ell(Q)}|m_{V_Q^1} u_a - m_{V_Q^2} u_a|\, d\lambda(a)\, dX \\
                                    & \lesssim   \int_{\Xi}  \sum_{Q \in \F_{3,m}(a)} \int_{V_Q}  \frac{1}{\ell(Q)}|m_{V_Q^1} u_a - m_{V_Q^2} u_a|\, dX\, d\lambda(a) \\
                                    & \qquad + \int_{\Xi}  \sum_{Q \in \F_{2,m}(a)} \int_{V_Q}  \frac{1}{\ell(Q)}|m_{V_Q^1} u_a - m_{V_Q^2} u_a|\, dX\, d \lambda(a) \\
                                    &\lesssim \int_{\Xi}  \sum_{Q \in \F_{3,m}(a)} \int_{V_Q}  \frac{1}{\ell(Q)}|m_{V_Q^1} u_a - m_{V_Q^2} u_a|\, dX\, d \lambda(a) \\
                                    & \qquad +\int_{\Xi} \sum_{Q \in \F_{2,m}(a)} \sigma(Q) \, d\lambda(a),
  \end{split}
\end{equation}
where the implicit constants depend on $M$, $n$ and the ADR constant and we used that $\lVert u \rVert_\infty \le 1$
in the last inequality.
We have by \eqref{AmpleSawtooth.eq} and the definitions of $\widetilde{\F}_{2,m}(a)$ and $\A(m)$ that
\begin{align}\label{restpack.eq}
  \begin{split}
    \frac{1}{\sigma(R)} \sum_{Q \in \F_{2,m}(a)} \sigma(Q) 
    &\le \frac{1}{\sigma(R)} \, \sum_{Q^* \in \widetilde{\F}_{2,m}(a)} \sum_{Q \in \bigcup_{k=0}^{m-1} \text{LD}^k(Q^*)} \sigma(Q) \\
    &\le\frac{1}{\sigma(R)} \sum_{Q^* \in \widetilde{\F}_{2,m}(a)} \sigma(Q^*)\A(m)
    \le \gamma \A(m).
  \end{split}
\end{align}
Set $B^{**}_R := B(x_R, 5a_1 \ell(R))$ and $\Delta^{**}_R = B^{**}_R \cap \om$ and note that $V_Q \subset B^{**}_R$ for all $Q \in \dd(R)$.
We immediately have that
\begin{align*}
\int_{B_R^{**} \cap \om} |\nabla \varphi(X) |\, dX &\lesssim \ell(R)^n \fint_{\Delta^{**}_R} \C(\nabla \varphi)(y)\, d
\sigma(y) 
\\& \approx \sigma(R)  \fint_{\Delta^{**}_R} \C(\nabla \varphi)(y)\, d\sigma(y).
\end{align*}
Using \eqref{epsinsawtooth.eq}, \eqref{VQoverlap.eq}, the Poincar\'e inequality\footnote{See e.g. the proof of \cite[Theorem 5.11.1]{Z}, for the case of BV functions.} and \eqref{CMEvarphi.eq} we obtain
\begin{equation}\label{goodpack.eq}
\begin{split}
&\int_{\Xi}  \sum_{Q \in \F_{3,m}(a)} \int_{V_Q}  \frac{1}{\ell(Q)}|m_{V_Q^1} u_a - m_{V_Q^2} u_a|\, dX\, d\lambda(a)
\\& \quad \lesssim \int_{\Xi}  \sum_{Q \in \F_{3,m}(a)} \int_{V_Q}  \frac{1}{\ell(Q)}|m_{V_Q^1} \varphi_a - m_{V_Q^2} 
\varphi_a| \, dX\, d\lambda(a) + \eps_0\int_{\Xi}   \sum_{Q \in \F_{3,m}(a)} \sigma(Q) 
\\& \quad {\lesssim \int_{\Xi}  \sum_{Q \in \F_{1,m}} \int_{V_Q} |\nabla \varphi_a | \, dX \, d\lambda + \eps_0 \sum_{Q 
\in \F_{1,m}} \sigma(Q) }
\\   &\quad  \lesssim  \int_{\Xi} \int_{B_R^{**} \cap \om}  |\nabla \varphi_a | \, dX \, d\lambda + \eps_0 \sum_{Q \in 
\F_{1,m}} \sigma(Q) 
\\  &\quad \lesssim  \int_{\Xi}  \fint_{\Delta^{**}} \C(\nabla \varphi)(y)\, d\sigma(y) \, d\lambda + \eps_0 \sum_{Q \in 
\F_{1,m}} \sigma(Q) 
\\ &\quad \lesssim  \int_{\Xi} \sigma(R) \left( \fint_{\Delta^{**}} (\C(\nabla \varphi_a)(y))^p\, d\sigma(y)\right)^{1/p} \, d
\lambda + \eps_0 \sum_{Q \in \F_{3,m}(a)} \sigma(Q) 
\\ & \quad \lesssim  C_{\eps_1}C_1\sigma(R)  + \eps_0 \sum_{Q \in \F_{3,m}} \sigma(Q),
\end{split}
\end{equation}
where the {\it implicit} constant depends only on $M$, $n$ and the ADR constant.
Dividing \eqref{bigest1.eq} by $\sigma(R)$ and using \eqref{restpack.eq} and \eqref{goodpack.eq} we have shown
$$\A(R,m) \lesssim C_{\eps_1}C_1 + \eps_0\A(R,m) + \gamma \A(m),$$
where the implicit constant depends only on $M$, $n$ and the ADR constant. Taking the supremum over $R \in \dd$ 
and recalling that $\eps_0 = C_2\eps_1$ we have
$$\A(m) \lesssim C_{\eps_1}C_1 + C_2\eps_1\A(m) + \gamma \A(m).$$
We recall the order we have chosen the constants and first choose $\gamma \ll 1$; this choice dictates $C_2$ ($C_2$ 
depends on $M$, $p$, $\gamma$, $n$, the ADR constant and $C_1$). Finally, we choose $\eps_1 \ll 1$ depending on 
$C_2$, $M$, $n$ and the ADR constant. Thus,
$$\A(m) \le C_{M,n,ADR} C_{\eps_1}C_1,$$
which allows us to conclude the claim. Here we have used the fact that the quantity $\A(m)$ is finite (to be more precise, we have 
$\A(m) \le m + 1< \infty$), which allows us to choose $\gamma$ and $\eps_1$ to be so small that $\A(m) - c(C_2 \eps_1 + \gamma) \A(m) > 0$
for a structural constant $c$ which was implicit in the estimates.
This concludes the proof.
\end{proof}

\section*{Acknowledgments}
The authors would like to thank Steve Hofmann and Svitlana Mayboroda for their mentorship and advice during the preparation of this article.

\end{document}